\documentclass[10pt, conference]{IEEEtran}
\usepackage{makeidx}
\usepackage{amssymb,amsmath,amsbsy,mathdots}
\usepackage{graphicx}
\newcommand{\dddots}{\begin{parbox}{0cm}{\scriptsize$\ddots$}\end{parbox}}
\usepackage{url}
\begin{document}
\title{Using cylindrical algebraic decomposition and local Fourier analysis to study numerical methods: two examples}
\author{\IEEEauthorblockN{Stefan Takacs}
\IEEEauthorblockA{Faculty for Mathematics, \\Research Group
Numerical Mathematics (Partial Differential Equations),\\
TU Chemnitz, Germany\\
Email: stefan.takacs@numa.uni-linz.ac.at}}
\thanks{The research was funded by the Austrian Science Fund (FWF): J3362-N25.}
\newtheorem{lemma}{Lemma}
\newcommand{\ee}{\textnormal{e}}
\newcommand{\ii}{\textnormal{i}}
\newcommand{\ul}[1]{\underline{#1}}
\newcommand{\mc}[1]{\mathcal{#1}}
\maketitle
\begin{abstract}
Local Fourier analysis is a strong and well-established tool for analyzing the convergence
of numerical methods for partial differential equations. The key idea of local Fourier
analysis is to represent the occurring functions in terms of a Fourier series and to
use this representation to study certain properties of the particular numerical method,
like the convergence rate or an error estimate.

In the process of applying a local Fourier analysis, it is 
typically necessary to determine the supremum of a more or less complicated term
with respect to all frequencies and, potentially, other variables. The problem of computing
such a supremum can be rewritten as a quantifier elimination problem, which can be solved
with cylindrical algebraic decomposition, a well-known tool from symbolic computation.

The combination of local Fourier analysis and cylindrical algebraic decomposition
is a machinery that can be applied to a wide class
of problems. In the present paper, we will discuss two examples. The first example
is to compute the convergence rate of a multigrid method. As second example we will see
that the machinery can also be used to do something rather different: We will compare
approximation error estimates for different kinds of discretizations.
\end{abstract}


\IEEEpeerreviewmaketitle

\section{Introduction}

In this paper, we want to give some examples where the combination of
cylindrical algebraic decomposition (CAD), as a tool from symbolic computation,
and local Fourier analysis (LFA) yield helpful results.
LFA was introduced by A. Brandt, who
proposed to use Fourier series to analyze multigrid methods, cf.
\cite{Brandt:1977}. For a detailed introduction into LFA, see, e.g.,~\cite{Trottenberg:2001}.
LFA provides a framework
to determine sharp bounds for the convergence rates of multigrid methods and
other iterative solvers for problems arising from partial differential equations.
This is different to classical analysis, which
typically  yields qualitative statements only.
So classical convergence proofs for multigrid solvers, cf.~\cite{Hackbusch:1985},
show that the method is convergent and that the convergence rates are uniformly
bounded away from $1$ for all grid sizes, however there is no sharp, nor realistic
bound for the convergence rate given.
Besides the analysis of linear solvers, the idea of LFA can be carried over to other 
applications, like the computation of approximation error estimates or the computation
of inverse inequalities.

LFA can be justified rigorously only in special
cases, e.g., on rectangular domains with uniform grids and periodic
boundary conditions. However, results obtained with 
LFA can be carried over to more general cases, see,
e.g.~\cite{Brandt:1994}. In cases, where such a extension is not possible,
it can be seen as heuristic approach.

To compute the quantities of interest using LFA, typically one has to compute the
supremum of a more or less complicated term. The key for involving symbolic algorithms is a proper
reformulation of the problem of computing a supremum as a quantifier elimination
problem, which can be solved using a CAD algorithm, cf.~\cite{Collins:1975}.
Understanding the combination of LFA and CAD
as a machinery for analyzing a numerical method, we apply this machinery in the present
paper to two examples, keeping in mind that there are more.

The first example is related to the classical idea of analyzing multigrid solvers.
In Sec.~\ref{sec:1}, we will introduce a classical finite element framework for
the Laplace equation and analyze a standard Jacobi iteration for solving the discretized
system. There, we will introduce the reader to the finite element method to keep the paper
readable also for non-numerical analysts. In Sec.~\ref{sec:2},
we will extend the analysis to be able to learn about convergence properties of
a multigrid solver. The given example is rather simple (and could
be solved also without use of CAD, just per hand).
However, we refer to other examples, where the terms get much more complicated,
which make symbolic tools more interesting, cf., e.g.,~\cite{Pillwein:Takacs:2011}
and~\cite{Pillwein:Takacs:2012}.

The second example, which will be discussed in Sec.~\ref{sec:3}, is a new result. It is given to show
that the machinery of LFA can also be extended to analysis beyond
analyzing the convergence of a multigrid solver. We will see that the method can
also be used to develop approximation error estimates. Moreover, we will see that LFA can
capture any kind of discretization. To keep it simple, we will stay in the one dimensional
case, so the terms, that have to be resolved using CAD,
are rather easy. We will provide supplementary material that covers also the extension to two
dimensions. There, one can see that in this case the terms get much more complicated.

This list of examples is not complete. So, CAD has already been applied 
earlier in the analysis of (systems of) ordinary and partial differential-difference
equations,~\cite{Hong:Liska:Steinberg}, where the necessary
conditions for stability, asymptotic stability and well-posedness of the given systems
were transformed into statements on polynomial inequalities using Fourier or Laplace
transforms.

\section{Finite element method and a simple iteration scheme}\label{sec:1}

We start our analysis with a simple example, the Laplace equation. For a given function
$f$, we are interested in finding a function $u$ such that
\begin{equation}\label{eq:laplacestrong}
      -u''(x) = f(x) 
\end{equation}
is satisfied for all $x\in \Omega:=(0,1)$ and, moreover, the boundary condition $u(0)=u(1)=0$
holds.

The standard way of solving this, is to introduce a variational formulation.
Let $H^1(\Omega)$ be the standard Sobolev space of weakly differentiable functions
and $H^1_0(\Omega)\subset H^1(\Omega)$ be the space of functions that moreover
satisfy the boundary condition $u(0)=u(1)=0$.
Then, the strong formulation~\eqref{eq:laplacestrong} can be rewritten in weak formulation
as follows: Find $u\in V:=H^1_0(\Omega)$ such that
\begin{equation}\label{eq:var}
    \int_{\Omega} u'(x) v'(x) \textnormal{d}x = \int_{\Omega} f(x) v(x) \textnormal{d}x
\end{equation}
for all $v\in V$, cf. standard literature on finite elements, like~\cite{Brenner:Scott:1994}.

For any finite dimensional subset $V_k \subset V$, we can introduce a discretized problem:
Find $u_k\in V_k$ such that
\begin{equation}\label{eq:variational2}
    \int_{\Omega} u_k'(x) v_k'(x) \textnormal{d}x = \int_{\Omega} f(x) v_k(x) \textnormal{d}x
\end{equation}
for all $v_k\in V_k$. The approach to use the same space, $V_k$, for both, $u_k$ and $v_k$,
is called the Galerkin principle. This guarantees that $u_k$ is the orthogonal projection of the
exact solution $u\in V$ into $V_k$.

The easiest way to set up the space $V_k$ is to choose the Courant element: Here the domain
$\Omega$ is subdivided into intervals (in one dimension) or into triangles (in two dimensions).
We call these intervals or triangles \emph{elements}.
The space $V_k$ consists of all globally continuous functions that are linear on each element.

Each function in $V_k$ can be characterized just by prescribing its values on the end points
of the intervals or at the vertices of the triangles, respectively -- we call these points \emph{nodes}. This
fact can be used to construct a basis: The nodal basis of $V_k$ is the collection of all functions
$\varphi_{k,i} \in V_k$ that take the value $1$ on exactly one of the nodes and the value
$0$ on all of the other nodes. One such basis function is visualized in Fig.~\ref{fig:1}.

\begin{figure}
\begin{center}
  \includegraphics[scale=.5]{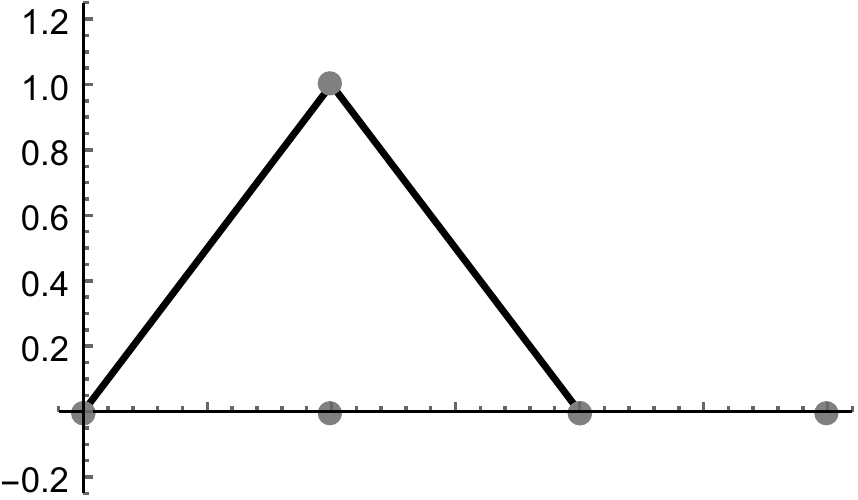}
\end{center}
\caption{Basis functions of standard Courant element}\label{fig:1}
\end{figure}

Having this basis, we can represent the functions $u_k$ and $v_k$ in terms of the basis:
\begin{equation*}
    u_k(x) = \sum_{i=1}^N u_{k,i} \varphi_{k,i}(x), \qquad v_k(x) = \sum_{i=1}^N u_{k,i} \varphi_{k,i}(x),
\end{equation*}
where the functions $u_k$ and $v_k$ can be represented by the coefficient vectors
$\ul{u}_k:=(u_{k,i})_{i=1}^N$ and $\ul{v}_k:=(v_{k,i})_{i=1}^N$. The variational
equality~\eqref{eq:variational2} can be rewritten in matrix-vector notation as follows:
\begin{equation}\label{eq:variational3}
    \ul{v}_k^T K_k \ul{u}_k = \ul{v}_k^T \ul{f}_k,
\end{equation}
for all $\ul{v}_k\in \mathbb{R}^N$,
where $K_k := (\int_{\Omega} \varphi_{k,j}'(x) \varphi_{k,i}'(x) \textnormal{d}x)_{i,j=1}^N$
and $\ul{f}_k := (\int_{\Omega} f_k(x) \varphi_{k,i}(x) \textnormal{d}x)_{i=1}^N$.
As \eqref{eq:variational3} is supposed to be satisfied for all $\ul{v}_k$,
it can be rewritten as follows: Find $\ul{u}_k $ such that
\begin{equation}\label{eq:matrvec}
    K_k \ul{u}_k = \ul{f}_k.
\end{equation}
To obtain a good approximation, it is often necessary to refine the intervals (or triangles)
used for the discretization of the partial differential equation.
In this case both, the number of unknowns and the condition number 
of the matrix $K_k$, grow. However, $K_k$ has a nice property: it
symmetric and positive definite.

A simple linear iteration scheme to solve a matrix-vector problem \eqref{eq:matrvec}
for $K_k$ being symmetric and positive definite, is
the (damped) Jacobi iteration. Assuming $\ul{u}_k^{(0)}$ to be some starting value,
the iteration procedure is given by
\begin{equation*}
      \ul{u}_k^{(m+1)} := \ul{u}_k^{(m)} + \tau (\mbox{diag}K_k)^{-1} ( \ul{f}_k - K_k \ul{u}_k^{(m)}),
\end{equation*}
where $\tau > 0$ is a given damping parameter. For $\tau = 1$, we obtain the standard Jacobi iteration.

As a next step, we are interesting in analyzing the convergence of the Jacobi
iteration scheme. So, using the exact solution $\ul{u}_k^*:= K_k^{-1} \ul{f}_k$, we obtain
\begin{equation*}
      \ul{u}_k^{(m+1)}-\ul{u}_k^* = (I - \tau (\mbox{diag}K_k)^{-1} K_k )(\ul{u}_k^{(m)}-\ul{u}_k^*)
\end{equation*}
and further
\begin{equation*}
      \|\ul{u}_k^{(m+1)}-\ul{u}_k^*\|_{K_k} \le \| I - \tau (\mbox{diag}K_k)^{-1} K_k\|_{K_k} \|\ul{u}_k^{(m)}-\ul{u}_k^*\|_{K_k},
\end{equation*}
where $\mc{S}_k:=I - \tau (\mbox{diag}K_k)^{-1} K_k$ is called the iteration matrix and $\|\cdot\|_{K_k}$ is
the vector norm $\|\ul{v}_k\|_{K_k}:=(\ul{v}_k^T K_k \ul{v}_k)^{1/2}$ or the associated matrix norm.
We have
\begin{equation*}
      \|\mc{S}_k\|_{K_k} = \| K_k^{1/2} (I - \tau (\mbox{diag}K_k)^{-1} K_k) K_k^{-1/2} \|,
\end{equation*}
where $\|\cdot\|$ is the standard Euclidean norm. As $K_k^{1/2} (I - \tau (\mbox{diag}K_k)^{-1} K_k) K_k^{-1/2}$
is symmetric, obtain further
\begin{equation*}
      \|\mc{S}_k\|_{K_k} = \rho( K_k^{1/2} (I - \tau (\mbox{diag}K_k)^{-1} K_k) K_k^{-1/2} ) = \rho(\mc{S}_k),
\end{equation*}
where $\rho(\cdot)$ is the spectral radius.

To determine the spectral radius, we use LFA: We compute the spectral radius
of $S_k$ explicitly for a special case. We assume to have
\begin{itemize}
	\item an infinitely large domain $\Omega$ (this neglects all influence coming from the boundary of the domain),
\end{itemize}
which is
\begin{itemize}
	\item discretized using an uniform (equidistant) grid.
\end{itemize}
For simplicity, here, we restrict ourselves to the one dimensional case. However,
LFA can also be worked out for two or more dimensions, cf.~\cite{Trottenberg:2001}.

For such an equidistant grid, we can  compute the stiffness matrix $K_k$ explicitly:
\begin{equation*}
  K_k = \frac{1}{h_k}
    \left(
      \begin{array}{cccccc}
  \dddots & \dddots\\
  \dddots & 2  & -1 & \\
	  & -1 & 2  & -1  \\
	  &    & -1 & 2  & -1 &  \\
	  &    &    & -1 & 2  & \dddots\;\; \\
	  &    &    &    & \dddots & \dddots\;\;
      \end{array}
    \right),
\end{equation*}
where $h_k$ is the grid size (length of the intervals).
	
As next step, we define for any frequency $\theta\in [0,2\pi)^d$ a vector of complex exponentials
\begin{equation*}
	\ul{\phi}_k(\theta) := ( \phi_{k,j}(\theta) )_{j\in \mathbb{Z}}:= ( \ee^{j \theta \ii})_{j\in \mathbb{Z}}
\end{equation*}
and observe that
\begin{equation}\label{eq:symb:K}
	K_k \ul{\phi}_k(\theta) = \underbrace{\frac{1}{h_k} ( - \ee^{-\theta \ii} + 2 
	- \ee^{\theta \ii}  )}_{\widehat{K_k}(\theta):=} \ul{\phi}_k(\theta)
\end{equation}
is satisfied, i.e., that $\ul{\phi}_k(\theta)$ is an eigenvector of $K_k$. In the LFA world,
the eigenvalue $\widehat{K_k}(\theta)$ is also called the symbol of $K_k$.

Based on the symbol of $K_k$, we can determine the symbol (eigenvalue) of the iteration
matrix $\mc{S}_k$. First note that $\mbox{diag}K_k = \tfrac{2}{h_k}I$ and
therefore $\widehat{\mbox{diag}K_k}(\theta) = \tfrac{2}{h_k}$. So, we obtain
\begin{align}
    \widehat{S_k}(\theta) &= 1 - \tau \frac{h_k}{2} \widehat{K_k}(\theta) \nonumber\\
      &= 1-\frac{\tau}{2} ( - \ee^{-\theta \ii} + 2 - \ee^{\theta \ii}  )
      &= 1-\tau ( 1 - \cos\theta  ).\label{eq:jsymb}
\end{align}
As we have mentioned above, we are interested in $\rho(\mc{S}_k)$. This spectral
radius can be expressed using the symbol:
\begin{align*}
      q(\tau):=\rho(\mc{S}_k ) = \hspace{-.1em} \sup_{\theta\in[0,2\pi)} |\widehat{S_k}(\theta )| 
	       = \hspace{-.1em}\sup_{\theta\in[0,2\pi)} |1-\tau ( 1 - \cos\theta  )|.
\end{align*}
By substituting the variable $\theta$ by $c:=\cos \theta$, we can completely 
eliminate the occurrence of trigonometric functions and obtain 
\begin{align*}
	 q(\tau):= \sup_{-1 \le c \le 1} |1-\tau ( 1 - c  )|.
\end{align*}
By definition, the supremum
is smallest upper bound, i.e., the smallest $\lambda$ such that
\begin{equation}\label{eq:quantif}
    \forall_{-1\le c \le 1} -\lambda \le 1-\tau(1-c) \le  \lambda.
\end{equation}
To determine the smallest $\lambda$ satisfying~\eqref{eq:quantif}, we have
to eliminate the quantifiers, i.e. to solve a quantifier elimination problem.

A quantifier elimination problem is the problem to find a quantifier free
formula that is equivalent to a quantified formula:
\begin{center}
	\fbox{
		\begin{minipage}[c][6.2em]{0.40 \textwidth}
			\emph{Quantified formula:}\\[.3em]
			$
				(Q_1)_{x_1}\,\ldots (Q_n)_{x_n}\, A(x_1,\ldots,x_n,y_1\ldots,y_m),
			$\\[.3em]
			where $Q_i\in\{\exists,\forall\}$ and
			$A$ is a finite boolean combination of
			polynomial inequalities
		\end{minipage}		
	}\\[1em]
	$\Leftrightarrow$\\[1em]
	\fbox{
		\begin{minipage}[c][5.2em]{0.40 \textwidth}
			\emph{Quantifier free formula:}
			$
				B(y_1\ldots,y_m),
			$\\[.3em]
			where $B$ is a finite boolean combination of
			polynomial inequalities.
		\end{minipage}		
	}
\end{center}
The solution of such a problem is possible using CAD, cf.~\cite{Collins:1975,Strzebonski:2000}.
By applying a CAD algorithm to \eqref{eq:quantif}, we obtain
\begin{align}
    & (\tau\le 0 \wedge \lambda \ge 1-2\tau) \vee
    (0<\tau\le 1 \wedge \lambda \ge 1)\nonumber\\
    & \vee (\tau> 1 \wedge \lambda \ge -1+2\tau)\label{eq:quantif2}
\end{align}
Here, the smallest $\lambda$ satisfying~\eqref{eq:quantif2} is piecewise
given by the terms $1-2\tau$, $1$ and $-1+2\tau$. So, we obtain
\begin{equation*}
    q(\tau) = \left\{ \begin{array}{ll}
	1-2\tau & \mbox{ for } \tau\le 0\\
	1	& \mbox{ for } 0 < \tau \le 1 \\
	-1+2\tau& \mbox{ for } 1 < \tau.
	\end{array}\right.
\end{equation*}
\begin{figure}
\begin{center}
  \includegraphics[scale=.5]{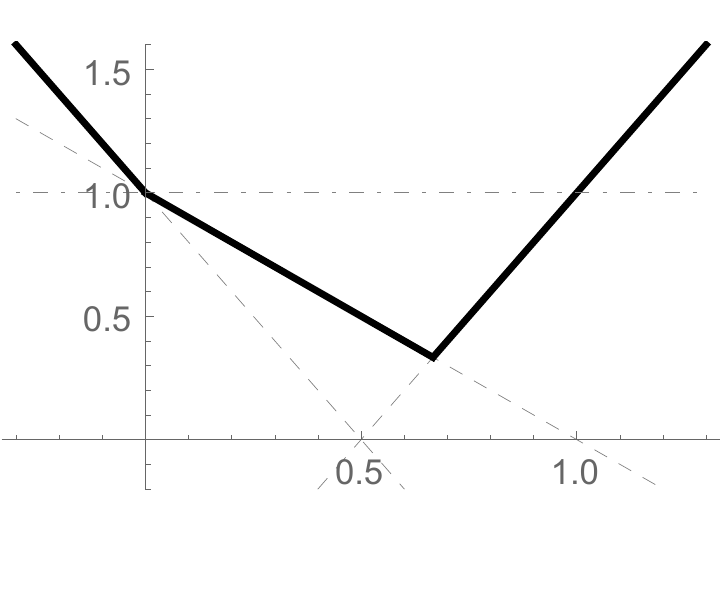}
\end{center}
\caption{Reduction of the high frequency modes as function of $\tau$}\label{fig:1a}
\end{figure}

We observe that there is no choice of $\tau$ such that $q(\tau) <1$. This reflects
knowledge on the Jacobi iteration (which is also true for other simple
linear iteration schemes): the convergence is not robust in the grid size $h_k$, so
the convergence rate cannot be bounded away from $1$. (Although, we did not have
an explicite dependence on the grid size $h_k$, the fact that we have considered an
unbounded domain $\Omega$ is equivalent to considering an infinitely small gird size.)

It is known by intuition that simple linear iteration schemes reduce high frequency error
modes. This statement can be formally expressed using LFA: Here, we
only consider $\theta \in [0,\pi/2) \cup [3\pi/2,\pi)$ or, equivalently, $0\le c \le 1$.
In this case, we obtain using the same arguments as above
\begin{align*}
	q_{SM}(\tau) =\sup_{0 \le c \le 1} |1-\tau ( 1 - c  )|
\end{align*}
Again, we can compute using CAD (or still per hand) that
\begin{align*}	
	q_{SM}(\tau)= \left\{ \begin{array}{ll}
	1-2\tau & \mbox{ for } \tau\le 0\\
	1-\tau	& \mbox{ for } 0 < \tau \le \tfrac23 \\
	-1+2\tau& \mbox{ for } \tfrac23 < \tau.
	\end{array}\right.
\end{align*}
This function is visualized in Fig.~\ref{fig:1a}. We see that $q_{SM}$ takes its minimal
value $\tfrac13$ for $\tau=\tfrac23$.

\section{Analysis of a multigrid solver}\label{sec:2}

In the last section, we have seen that the Jacobi iteration reduces the high frequency error modes.
The idea of a multigrid method is to use the fact that low frequency error modes can be
resolved well also on a coarse grid. So, we combine the Jacobi iteration (or any other simple linear
iteration scheme) with a coarse grid correction, which reduces the low frequency error modes. 

We assume to have for $k=1,2,3,\ldots$ a hierarchy of grid levels, where a grid level $k$
is obtained from grid level $k-1$ by uniform refinement, i.e., in the case of one dimension:
by subdividing each interval into two equally sized intervals. Starting from an
iterate~$\ul{x}^{(m)}_k$, the next iterate~$\ul{x}^{(m+1)}_k$ of the multigrid method on
grid level $k$ is given by the following three steps:
\begin{itemize}
        \item \emph{Pre-Smoothing:} Compute
              \begin{equation} \nonumber
                   \ul{u}^{(m,1)}_k := \ul{u}^{(m)}_k + \tau (\mbox{diag } K_k)^{-1}
                                    \left(\ul{f}_k -K_k\;\ul{u}^{(m)}_k\right).
              \end{equation}
        \item \emph{Coarse-grid correction:}
                \begin{itemize}
                     \item Compute the defect 
                        $\ul{ f}_k -K_k\;\ul{u}^{(m,1)}_k$
                        and restrict it to grid level $k-1$:\;
                       \begin{equation}\nonumber
                              \ul{r}_{k-1}^{(m)} := P_{k-1}^T \left(\ul{ f}_k -K_k
                              \;\ul{u}^{(m,1)}_k\right).
                       \end{equation}
                     \item Solve the following coarse-grid problem approximatively:
                        \begin{equation}\label{eq:coarse:grid:problem}
                            K_{k-1} \,\ul{p}_{k-1}^{(m)} =\ul{r}_{k-1}^{(m)}.
                        \end{equation}
                     \item Prolongate $\ul{p}_{k-1}^{(m)}$  to the
                          grid level $k$ and add
                          the result to the previous iterate:
                          \begin{equation}\nonumber
                               \ul{u}_{k}^{(m,2)} := \ul{u}^{(m,1)}_k +
                                P_{k-1} \, \ul{p}_{k-1}^{(m)}.
                          \end{equation}
                \end{itemize}
        \item \emph{Post-Smoothing:} Compute
              \begin{equation} \nonumber
                   \ul{u}^{(m+1)}_k := \ul{u}^{(m,2)}_k + \tau (\mbox{diag } K_k)^{-1}
                                    \left(\ul{f}_k -K_k\;\ul{u}^{(m,2)}_k\right).
              \end{equation}
\end{itemize}
As we have nested spaces, i.e., $V_{k-1}\subseteq V_k$, there is canonical
embedding from $V_{k-1}$ into $V_k$, which is chosen as prolongation
operator $P_{k-1}$.

If the problem~\eqref{eq:coarse:grid:problem} is solved exactly,
we obtain the two-grid method.
In practice, the problem~\eqref{eq:coarse:grid:problem} is
approximatively solved by applying one step (V-cycle)
or two steps (W-cycle) of the multigrid method, recursively. Only on
the coarsest grid level,~\eqref{eq:coarse:grid:problem} is 
solved exactly.

For computing the convergence rate of the multigrid solver,
we set up again the iteration matrix $\mc{G}_k$, which is the product of the iteration
matrix $\mc{S}_k$ of the damped Jacobi iteration,
of the iteration matrix $\mc{C}_k$ of the coarse-grid correction and, once more,
of the iteration matrix $\mc{S}_k$ of the damped Jacobi iteration:
\begin{equation*}
	\mc{G}_k = \mc{S}_k \mc{C}_k \mc{S}_k,
\end{equation*}
where
\begin{equation*}
	\mc{C}_k = I-P_{k-1} K_{k-1}^{-1}P_{k-1}^T K_k
\end{equation*}
and, as in the last section,
\begin{equation*}
	\mc{S}_k = I-\tau(\mbox{diag }K_k)^{-1} K_k.
\end{equation*}
As in the last section, we are interested in computing 
\begin{equation*}
	q(\tau) = \|\mc{G}_k\|_{K_k} = \rho(\mc{G}_k).
\end{equation*}

\begin{figure}
\begin{center}
  \includegraphics[scale=.5]{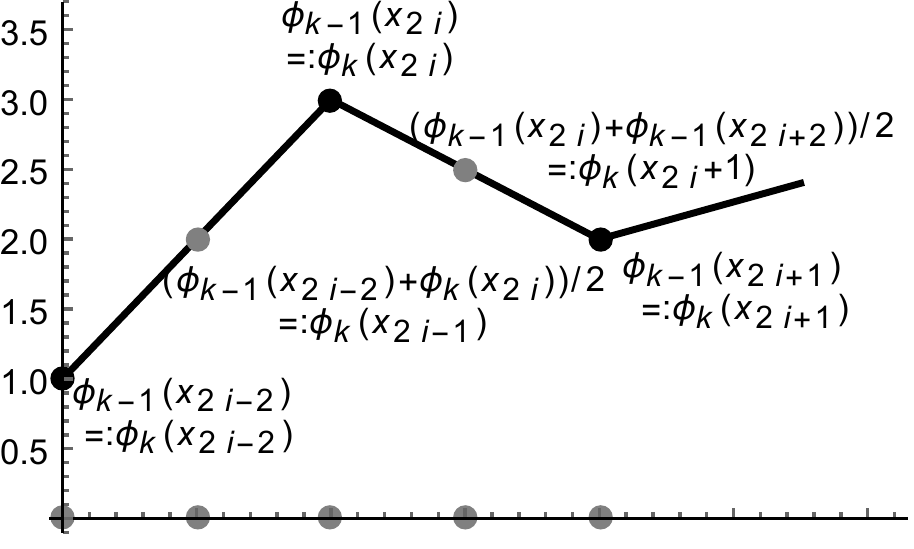}
\end{center}
\caption{Canonical embedding of $V_{k-1}$ into $V_k$}\label{fig:prolong}
\end{figure}

To be able to determine the symbol of the iteration matrix $\mc{G}_k$, we have to
take a closer look onto the prolongation operator $P_{k-1}$ first.
We recall that there is an isomorphism between $\mathbb{R}^N$, the space of
coefficient vectors, and the function space $V_k$. So, 
for each coefficient vector $\ul{\phi}_k(\theta)=(\phi_{k,j}(\theta))_{j\in\mathbb{Z}}$,
there is a function $\phi_{k}(\theta,\cdot)\in V_k$, which is assigned to it:
\begin{equation*}
  \phi_{k}(\theta,x)=\sum_{j\in\mathbb{Z}} \phi_{k,j}(\theta)\varphi_{k,j}(x).
\end{equation*}
By definition, $P_{k-1}$ is the canonical embedding operator, which is visualized in 
Fig.~\ref{fig:prolong}.

The next step is to represent the function $\phi_{k-1}(2\theta,x)$ as a linear combination
of functions on the fine grid. We observe, that this can be done using
the ansatz 
\begin{equation*}
      \phi_{k-1}(2\theta,x) = A \phi_{k}(\theta,x) + B\phi_{k}(\theta+\pi,x).
\end{equation*}
It is sufficient to consider the nodes $x_j = j h_k$ only. First we consider the
even nodes $x_{2j}$, which are also nodes of the coarse grid:
\begin{equation}\label{eq:xx1}
      \phi_{k-1}(2\theta,x_{2j}) = A \phi_{k}(\theta,x_{2j}) + B\phi_{k}(\theta+\pi,x_{2j}).
\end{equation}
As the $(\varphi_{k,i})_{i\in\mathbb{Z}}$, form a nodal basis, \eqref{eq:xx1} is equivalent to
\begin{equation}\nonumber
      \phi_{k-1,j}(2\theta) = A \phi_{k,2j}(\theta) + B\phi_{k,2j}(\theta+\pi),
\end{equation}
\begin{equation}\nonumber
      \ee^{j 2\theta  \ii} = A \ee^{2j \theta \ii} + B\ee^{2j(\theta+\pi) \ii}
\end{equation}
and, finally,
\begin{equation}\nonumber
     1 = A  + B.
\end{equation}
Now, we consider the odd nodes $x_{2j+1}$, which do not occur on the coarse grid:
\begin{equation}\label{eq:yy1}
      \phi_{k-1}(2\theta,x_{2j+1}) = A \phi_{k}(\theta,x_{2j+1}) + B\phi_{k}(\theta+\pi,x_{2j+1}).
\end{equation}
As the $(\varphi_{k,i})_{i\in\mathbb{Z}}$, form a nodal basis, \eqref{eq:yy1} is equivalent to
\begin{align*}
      &\frac12\left(\phi_{k-1,j}(2\theta)+\phi_{k-1,j+1}(2\theta)\right) \\
      &\qquad = A \phi_{k,2j+1}(\theta) + B\phi_{k,2j+1}(\theta+\pi)
\end{align*}
and
\begin{equation}\nonumber
     \frac12\left( \ee^{j2\theta  \ii}+\ee^{2 (j+1)\theta \ii} \right) = A \ee^{ (2 j+1)\theta \ii} + B\ee^{(2 j+1)(\theta+\pi)  \ii}
\end{equation}
and, finally,
\begin{equation}\nonumber
     \underbrace{\frac12\left( \ee^{-\theta \ii}+\ee^{\theta \ii} \right)}_{\cos(\theta)=} = A  - B.
\end{equation}
We obtain $A=\tfrac12 (1+\cos(\theta))$ and $B=\tfrac12 (1-\cos(\theta))$, which can be observed also
in Fig.~\ref{fig:intergrid}.
This allows to introduce the symbol of the prolongation operator:
\begin{align*}
 \widehat{P_{k-1}}(\theta) & = \frac{1}{2} \left(\begin{array}{c}
      1+\cos(\theta) \\
      1-\cos(\theta)
   \end{array} \right).
\end{align*}
\begin{figure}
\begin{center}
  \includegraphics[scale=.42]{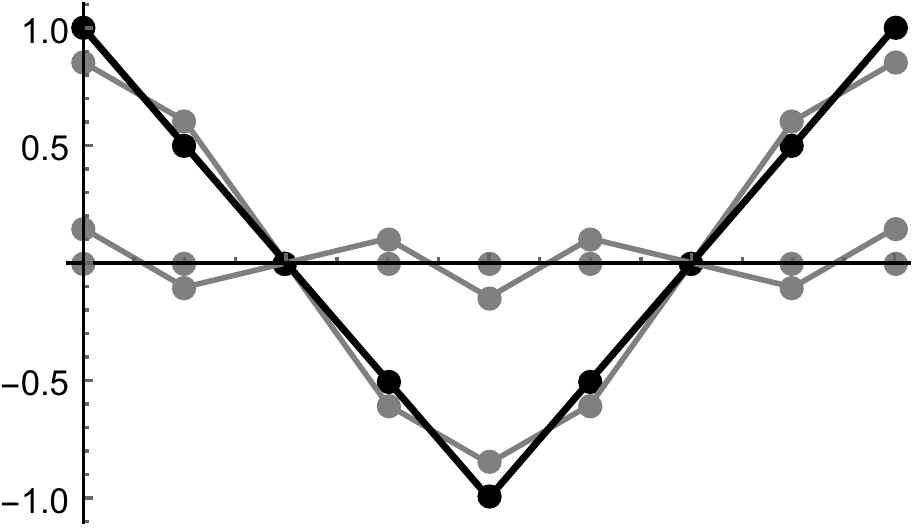}
\end{center}
\caption{Coarse-grid function $\phi_{k-1}(2\theta,x)$ in black and the two components
$\tfrac12(1+\cos(\theta)) \phi_k(\theta,x) $ and $\tfrac12(1-\cos(\theta)) \phi_k(\theta+\pi,x)$
in gray. }\label{fig:intergrid}
\end{figure}

Here, the symbol cannot be understood as eigenvalue anymore.
However, for all $\theta = [0,2\pi)$, the prolongation operator $P_{k-1}$ maps the linear
span, spanned by
\begin{equation}\label{eq:basis0}
  \ul{\phi}_{k-1}(2\theta)
\end{equation}
to the linear span, spanned by
\begin{equation}\label{eq:basis}
  \ul{\phi}_k(\theta) \quad \mbox{and}\quad  \ul{\phi}_k(\theta+\pi),
\end{equation}
and the restriction operator $P_{k-1}^T$ maps the linear
span, spanned by~\eqref{eq:basis}, to the linear span, spanned by~\eqref{eq:basis0}.

Having this, we can set up the symbol for the two-grid operator $\mc{G}_k$. We make
use of the fact that the multiplication of $\mc{G}_k$  with a vector in the linear span, given by
the basis~\eqref{eq:basis}, maps into the same linear span.
So, we have to set up
the symbol of~$\mc{G}_k$ with respect to the two dimensional basis~\eqref{eq:basis}. 

The symbol of $S_k$ has been a scalar in the last section. This means that
every frequency was preserved by the action of $S_k$. If we represent the symbol
of $S_k$ with respect to the basis~\eqref{eq:basis}, we just obtain a diagonal
symbol:
\begin{equation*}
	\widehat{\mc{S}_k}(\theta) = \left(
		\begin{array}{cc}
			\widehat{S_k}(\theta) \\ & \widehat{S_k}(\theta+\pi)
		\end{array}
	\right),
\end{equation*}
where $\widehat{S_k}(\theta)$ is as defined in~\eqref{eq:jsymb}.
Exactly the same way, we obtain the symbol $\widehat{\mc{K}_k}(\theta)$ based on
$\widehat{K_k}(\theta)$, given in~\eqref{eq:symb:K}.
Using this, we can determine the symbol of $\mc{C}_k$,
\begin{equation*}
    \widehat{\mc{C}_k}(\theta) = I -  \widehat{P_{k-1}}(\theta) [\widehat{K_{k-1}}(\theta)]^{-1}\widehat{P_{k-1}}(\theta)^*\widehat{\mc{K}_k}(\theta),
\end{equation*}
where $A^*$ is the conjugate complex of $A^T$. Consequently, the symbol of $\mc{G}_k$ is
\begin{equation*}
    \widehat{\mc{G}_k}(\theta) = \widehat{\mc{S}_k}(\theta)\widehat{\mc{C}_k}(\theta)\widehat{\mc{S}_k}(\theta).
\end{equation*}
Here, the computation of $\widehat{\mc{G}_k}(\theta)$ and of $\rho(\widehat{\mc{G}_k}(\theta))$ is straight-forward.
We obtain:
\begin{equation*}
    \rho(\widehat{\mc{G}_k}(\theta)) = |(\tau-1)^2+\tau(3\tau-2)\cos^2(\theta)|.
\end{equation*}
As in the last section, we are again interested in computing the supremum
\begin{equation*}
    q(\tau) = \rho(\mc{G}_k) = \sup_{\theta\in[0,2\pi)} |(\tau-1)^2+\tau(3\tau-2)\cos^2(\theta)|,
\end{equation*}
where we again substitute $\cos\theta$ by $c$ and obtain
\begin{equation*}
    q(\tau) = \sup_{c\in[-1,1]} |(\tau-1)^2+\tau(3\tau-2)c^2|.
\end{equation*}
Also here, we can resolve the supremum using a CAD algorithm (or, still, per hand) and obtain
\begin{equation*}
    q(\tau) = \left\{
	\begin{array}{ll}
		1-4\tau+4\tau^2 & \mbox{ for } \tau < 0 \\
		1-2\tau+\tau^2 & \mbox{ for } 0\le\tau < \tfrac23 \\
		1-4\tau+4\tau^2 & \mbox{ for } \tfrac23 \le \tau. \\
	\end{array}
      \right.
\end{equation*}
This function is shown in Fig.~\ref{fig:1b}. We see that $q$ takes its minimal value
$\tfrac19$ for $\tau=\tfrac23$.
\begin{figure}
\begin{center}
  \includegraphics[scale=.5]{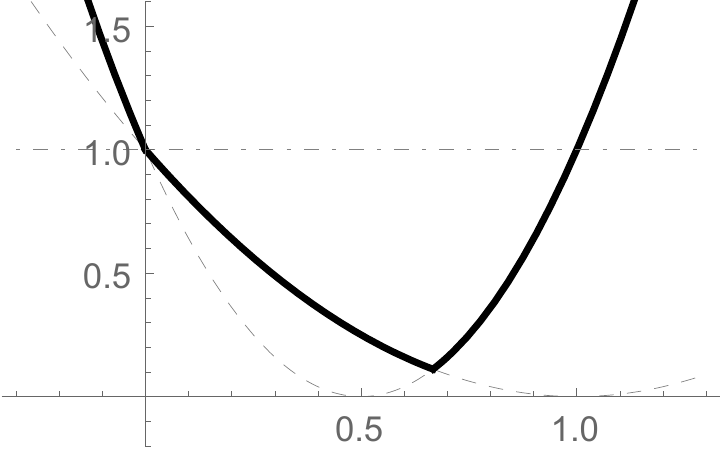}
\end{center}
\caption{Convergence rate of the multigrid solver as a function of $\tau$}\label{fig:1b}
\end{figure}

So far, all computations had been so easy such that it would have been possible to do them
per hand. However, the methodology presented in this section can be carried over to more complex
(and more interesting) problems. The first extension would be to consider two or
more dimensions. Here, one could represent everything use a tensor-product structure,
cf.~\cite{Trottenberg:2001}. Consequently, one has to deal with tuples of $d$ frequencies for
$d$ dimensional spaces. Also in this case, the $\theta_i$ can be substituted by $c_i:=\cos(\theta_i)$
and solved as discussed in this session. However, the complexity of the expressions (particularly in
terms of the polynomial degree) grows very fast if $d$ is increased.

Besides that, the presented methodology can be extended to non-standard problems. This is 
of practical use because the convergence analysis has to be worked out for each problem class,
separately. Here, LFA can be of great help.

One example where the presented approach has been applied in this fashion was in a
in a joint work with V.~Pillwein\footnote{Research Institute for Symbolic Computation,
Johannes Kepler University Linz, Austria}, cf.~\cite{Pillwein:Takacs:2011,Pillwein:Takacs:2012},
where LFA and CAD have been used to compute convergence rates of
a multigrid solver for a \emph{system of PDEs} which characterizes the solution of an
optimal control problem. There, not only the robustness of the convergence rates in the
grid size $h_k$, but also the robustness of the convergence rates in a regularization
parameter, which is part of the problem description, was of interest and could be studied.
The supplementary material, that came with the cited paper,
is available in the web\footnote{\url{http://www.risc.jku.at/people/vpillwei/sLFA/}}.
The author wants to refer the reader, which is interested in analyzing multigrid convergence,
to that material.

In the following of the present paper, the author wants to draw the reader's attention to another application of LFA
that is also of interest in numerical analysis: the estimation of approximation
error estimates.

\section{Estimate the approximation error}\label{sec:3}

In this section,
we are interested in comparing estimates of the approximation error
\begin{equation*}
	\inf_{u_k \in V_k} \|u -u_k\|_{L^2(\Omega)}
\end{equation*}
for different kinds of discretizations. One of the discretizations will be the Courant
element, two more will be introduced below. Here and in what follows $\|\cdot\|_{L^2(\Omega)}$
is the standard $L^2$-norm, i.e., $\|f\|_{L^2(\Omega)}^2:=\int_{\Omega} f^2(x) \mbox{d}x$

One important approximation error estimate reads as follows:
\begin{equation*}
	\inf_{u_k \in V_k} \|u -u_k\|_{L^2(\Omega)}^2 \le C_A h_k^2 |u|_{H^1(\Omega)}^2
\end{equation*}
for all $u\in L^2(\Omega)$, where $C_A>0$ is a constant, $h_k$ is the grid size and 
$|u|_{H^1(\Omega)}:=\|u'\|_{L^2(\Omega)}$.
For classical discretizations, it is well-known that such
an estimate exists. However, often there is no realistic bound for the constant $C_A$.
So, it might be of interest to compute an realistic
(not necessarily sharp) upper bound for the constant $C_A$ for discretizations of
interest.

The approximation error can be bounded from above using an interpolation
error $\|u -\Pi_k u\|_{L^2(\Omega)}$, where $\Pi_k:H^1(\Omega)\rightarrow V_k$ is an
arbitrarily projection operator. So, it suffices to estimate
\begin{equation}\label{eq:0}
	\|u -\Pi_k u\|_{L^2(\Omega)}^2 \le C_A h_k^2 |u|_{H^1(\Omega)}^2
\end{equation}
for any projection operator $\Pi_k$. Using the following lemma, we show
\eqref{eq:0} for $\Pi_k$ being the $H^1$-orthogonal projection.
\begin{lemma}
	Let for all grid levels $k\in\mathbb{N}$, the operator $\Pi_k$ be the $H^1$-orthogonal
	projection form $H^1(\Omega)$ into $V_k$.
	Assume that for all $k$ the following \emph{quantitative} estimate on two consecutive grids is satisfied:
	\begin{equation}\label{eq:1}
	    \| (I-\Pi_k) u_{k+1} \|_{L^2(\Omega)}^2 \le C_A h_k^2 |u_{k+1}|_{H^1(\Omega)}^2 
	\end{equation}
	for all $u_{k+1}\in V_{k+1}$.
	Moreover, we assume to know \emph{qualitatively} that
	\begin{equation}\label{eq:2}
	    \| (I-\Pi_k) u \|_{L^2(\Omega)} \rightarrow 0 \mbox{ for } k\rightarrow \infty
	\end{equation}
	for all $u \in L^2(\Omega)$. Then the following estimate is satisfied:
	\begin{equation*}
	    \| (I-\Pi_k) u \|_{L^2(\Omega)}^2 \le 4 C_A h_k^2 |u|_{H^1(\Omega)}^2
	\end{equation*}
	for all $u \in L^2(\Omega)$.
\end{lemma}
\begin{proof}
  The proof is based on a simple telescoping argument.
  Due to \eqref{eq:2}, for any $\epsilon > 0 $ there is some $K>0$ such that
  $\| (I-\Pi_K) u \|_{L^2(\Omega)} < \epsilon |u|_{H^1(\Omega)}$.
  Now, we obtain due to the triangular inequality, \eqref{eq:1} and the fact that the $H^1$-orthogonal projection
  is stable in $H^1(\Omega)$, i.e., $|\Pi_ku|_{H^1(\Omega)}\le|u|_{H^1(\Omega)}$,
  \begin{align*}
    & \| (I-\Pi_k) u \|_{L^2(\Omega)} \\
    & \le \| (I-\Pi_K) u \|_{L^2(\Omega)} + \sum_{m=k}^{K-1} \| (I-\Pi_m) \Pi_{m+1} u \|_{L^2(\Omega)} \\
    & \le \left( \epsilon + \sum_{m=k}^{K-1} C_A^{1/2} h_m \right) |u|_{H^1(\Omega)} =:\Psi.
  \end{align*}
  As $h_m = 2^{k-m} h_{k}$, we obtain using the summation formula for the geometric series that
    $\Psi \le (\epsilon + 2 C_A^{1/2} h_k) |u|_{H^1(\Omega)}$ and for $\epsilon\rightarrow 0$ the desired result.
\end{proof}

The statement \eqref{eq:2} is well-known for all standard discretizations. However, there might not be
a good estimate for $C_A$. So, we are interested in the results by this lemma. The estimate~\eqref{eq:1}
can be treated using LFA. We can rewrite~\eqref{eq:1} in matrix-vector notation
as follows:
\begin{equation*}
    \| (I-P_k K_k^{-1} P_k^T K_{k+1}) \ul{u}_{k+1} \|_{M_{k+1}}^2 \le C_A h_k^2 \|\ul{u}_{k+1}\|_{K_{k+1}}^2,
\end{equation*}
where $M_k :=(m_{i,j})_{i,j=1}^N:= (\int_{\Omega} \varphi_{k,j}(x) \varphi_{k,i}(x) \textnormal{d}x)_{i,j=1}^N$
is the mass matrix.  Here, the upper bound is obtained using the matrix norm:
\begin{equation*}
    C_A^{1/2} = \frac{1}{h_k} \left\|M_{k+1}^{1/2} (I-P_k K_k^{-1} P_k^T K_{k+1}) K_{k+1}^{-1/2} \right\| .
\end{equation*}
Using the definition of the Euclidean norm and the fact that $(I-P_k K_k^{-1} P_k^T K_{k+1})^2=(I-P_k K_k^{-1} P_k^T K_{k+1})$, we obtain
\begin{align*}
    C_A & = \frac{1}{h_k^2} \rho\big(\underbrace{ M_{k+1} (I-P_k K_k^{-1} P_k^T K_{k+1}) K_{k+1}^{-1} }_{\mc{G}_{k+1}:=}\big).
\end{align*}
Here, again, the spectral radius can be determined using the symbol
\begin{align}
    &C_A = \sup_{\theta\in[0,2\pi)}  \rho\Big( \widehat{\mc{G}_{k+1}}(\theta)\Big),\mbox{ where}\nonumber\\
    & \widehat{\mc{G}_{k+1}}(\theta):=\frac{1}{h_k^2}\widehat{\mc{M}_{k+1}}(\theta)  \widehat{\mc{C}_{k+1}}(\theta) \left(\widehat{\mc{K}_{k+1}}(\theta)\right)^{-1},\nonumber\\
    &\widehat{\mc{C}_{k+1}}(\theta):= \left(I-\widehat{P_k}(\theta) \left(\widehat{K_k}(\theta)\right)^{-1} \widehat{P_k}(\theta)^* \widehat{\mc{K}_{k+1}}(\theta)\right).\nonumber
\end{align}

As we have mentioned above, we are interested in computing~$C_A$ for different
discretizations. The details can be found in an accompanying Mathematica notebook,
which is available in the web\footnote{\url{http://www.numa.uni-linz.ac.at/~stefant/J3362/slfa/}},
the main ideas will be given in the following three subsections.

\subsection{The Courant element}

The symbols $\widehat{\mc{K}_{k+1}}(\theta)$ and $\widehat{P_k}(\theta)$ for the Courant element have already
been determined in the last section. The mass matrix $M_k$
has also a tridiagonal form. The symbol can be computed completely analogous as for the stiffness
matrix:
\begin{equation*}
    \widehat{\mc{M}_{k+1}}(\theta) =\left(\begin{array}{cc}\widehat{M_{k+1}}(\theta)\\&\widehat{M_{k+1}}(\theta+\pi)\end{array}\right),
\end{equation*}
where
\begin{equation*}
    \widehat{M_{k+1}}(\theta) = \frac16(\ee^{-\theta\ii}+4+\ee^{\theta\ii}).
\end{equation*}
Based on this, we can derive
\begin{equation*}
    \widehat{\mc{G}_{k+1}}(\theta) = \frac{1}{12} \left(
      \begin{array}{cc}
	  2 + \cos\theta & -2 + \cos\theta\\
	  -2 - \cos\theta & 2 - \cos\theta
      \end{array}
    \right).
\end{equation*}
The eigenvalues of $\widehat{\mc{G}_{k+1}}(\theta)$ are $0$ and $\tfrac13$. As this is already independent of
$\theta$, we immediately obtain that for the Courant element $C_A = \tfrac13$ is satisfied.

\subsection{A $P^2$-spline discretization}

We can set up the same framework also for other discretizations, like the discretization
with splines. Here, assume that $V_k$ is the space of all continuously
differentiable functions, which are piecewise polynomials of degree $2$. One possible
basis for $V_k$ is the basis of B-splines:
\begin{equation*}
    \varphi_{k,i}(x) =  \left\{
	\begin{array}{lr}
	      \tfrac{1}{2h_k^2} (x - x_{i-1})^2 & \mbox{for } x_{i-1} \le x < x_{i} \\
	      \tfrac{3}{4} - \tfrac{1}{4h_k^2}(2x - x_{i}-x_{i+1})^2\hspace{-.6cm} & \mbox{for } x_{i} \le x < x_{i+1} \\
	      \tfrac{1}{2h_k^2} (x - x_{i+2})^2 & \mbox{for } x_{i+1} \le x < x_{i+2}\\
	      0 & \mbox{otherwise,}
	\end{array}
    \right.
\end{equation*}
where $x_{i} = i h_k$, see Fig.~\ref{fig:2} for a visualization of such a basis function.
\begin{figure}
\begin{center}
  \includegraphics[scale=.5]{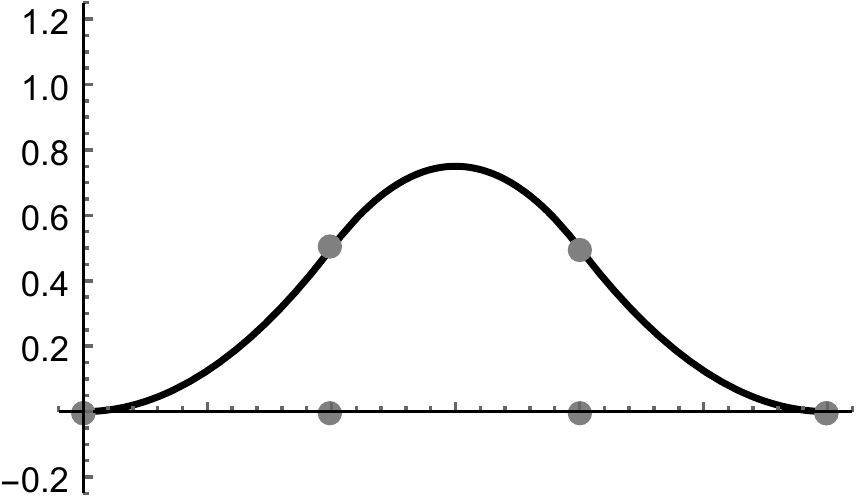}
\end{center}
\caption{Basis functions for the $P^2$-spline discretization}\label{fig:2}
\end{figure}

For the B-splines, we can again compute the integrals that are necessary to
set up the mass matrix $M_k$. As the support of the B-splines
is larger than the support of the basis functions of the Courant element, 
we obtain a band matrix with a bandwidth of $5$, with $m_{i,i}= \tfrac{66}{120}h_k$,
$m_{i,i\pm 1}= \tfrac{26}{120}h_k$ and $m_{i,i\pm 2}= \tfrac{1}{120}h_k$.
Also for this case, we can determine the symbol
\begin{equation*}
  \widehat{M_k}(\theta) = \frac{h_k}{120}\left(
    \ee^{-2\ii \theta} + 26\ee^{-\ii \theta} + 66 + 26\ee^{\ii \theta} +  \ee^{2\ii \theta}
    \right).
\end{equation*}
We can set up the the stiffness matrix $K_k$ and its symbol in a completely analogous way and obtain
\begin{equation*}
    \widehat{K_k}(\theta) = \frac{1}{6 h_k}\left(
   - \ee^{-2\ii \theta} -2\ee^{-\ii \theta} + 6 -2\ee^{\ii \theta} -  \ee^{2\ii \theta}
    \right).
\end{equation*}
For setting up the symbol of the prolongation operator $P_{k-1}$, it is sufficient
to solve again the equations~\eqref{eq:xx1} and~\eqref{eq:yy1}. For details, we refer
to the Mathematica notebook. The overall symbol $\widehat{\mc{G}_{k+1}}(\theta)$ is again
just obtained by multiplying the individual symbols. The eigenvalues
of $\widehat{\mc{G}_{k+1}}(\theta)$ are $0$ and
\begin{equation}\label{eq:largerev}
   \frac{-51 + 14 \cos(2 \theta) + \cos(4 \theta)}{40 (-2 + \cos(\theta)) (2 + \cos(\theta)) (2 + \cos(2 \theta))}.
\end{equation}
This second eigenvalue can be rewritten using the replacement $\cos\theta\rightarrow c$
as rational function, where the terms $\cos(2\theta)$ and $\cos(4\theta)$ are treated
using the corresponding Chebyshev polynomials. Here we obtain -- using CAD -- that $\tfrac25$
is the largest value taken by~\eqref{eq:largerev}, so we obtain $C_A=\tfrac25$. 

\subsection{A standard $P^2$-discretization}

Besides the spline functions, there is another possibility of setting up a discretization
based on polynomials of degree $2$, which is even more popular in finite elements:
we define $V_k$ to be the space of continuous functions that are piecewise polynomials
of degree $2$. Here, we can introduce a nodal basis, i.e., a basis where each basis
function is associated to node (this basis function takes the value $1$ on that node
and the value $0$ on all other nodes). Here, the nodes are allocated on the ends
of the intervals (as for the Courant element) and, additionally, on the midpoints of
the elements. Here, we have two types of basis functions, cf. Fig.~\ref{fig:3}
and Fig.~\ref{fig:3a} for visualizations.
\begin{figure}
\begin{center}
  \includegraphics[scale=.5]{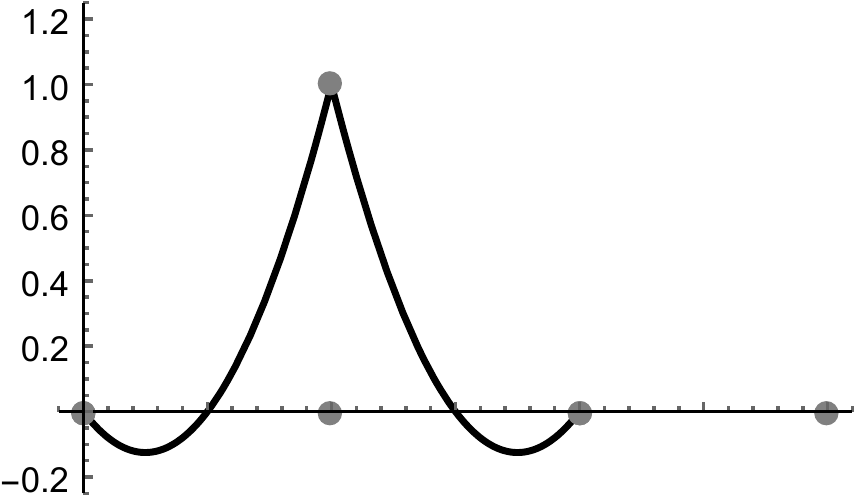}
\end{center}
\caption{Basis functions of the first kind of the $P^2$-discretization}\label{fig:3}
\end{figure}
\begin{figure}
\begin{center}
  \includegraphics[scale=.5]{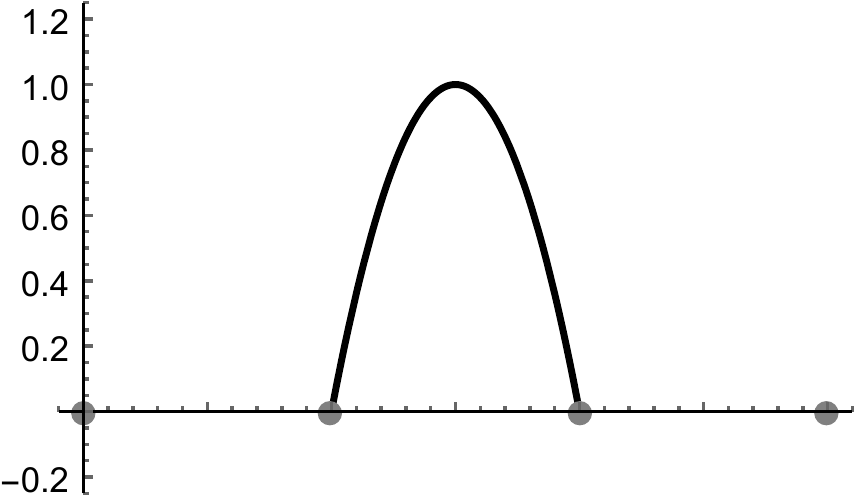}
\end{center}
\caption{Basis functions of the second kind of the $P^2$-discretization}\label{fig:3a}
\end{figure}

Because there are two types of elements, the mass matrix has alternating coefficients, see
the Mathematica notebook for details:
\begin{equation*}
  M_k = \frac{h_k}{30}
    \left(
      \begin{array}{cccccccc}
	  \dddots & \dddots & \dddots \\
	  \dddots & 8  & 2  & -1 \\
	  \dddots & 2  & 16 &  2  & 0\\
	          & -1 & 2  &  8 &  2 & -1 \\
	          &    &  0 &  2 & 16 &  2 & 0\\
	          &    &    & -1 &  2 &  8 & 2 & \dddots\;\\
	          &    &    &    &  0 &  2 & 16 & \dddots\; \\
	          &    &    &    &    &  \dddots  &  \dddots & \dddots\;
      \end{array}
    \right).
\end{equation*}
For determining the symbol of $M_k$, we rewrite $M_k$ as a sum of a band-matrix and of a
residual matrix with alternating signs:
\begin{align*}
 & M_k = A_k+B_k,
\end{align*}
where $A_k=(a_{i,j})_{i,j\in\mathbb{Z}}$ is a band matrix with $a_{i,i}=\tfrac{2}{5}h_k$, 
$a_{i,i\pm1}=\tfrac{1}{15}h_k$ and $a_{i,i\pm2}=-\tfrac{1}{60}h_k$ and $B_k$ is a matrix with
alternating coefficients:
\begin{align*}
 & B_k:=  \frac{h_k}{60} \left(
      \begin{array}{ccccccc}
  \dddots & \dddots & \dddots \\
  \dddots & -8 & 0  &-1\\
  \dddots &  0 & 8  & 0  &1\\
	  & -1  & 0  & -8 & 0  &-1\\
	  &    &1  & 0  & 8  & 0  & \dddots\;\; \\
	  &    &    & -1  & 0  & -8 & \dddots\;\;  \\
	  &    &    &    & \dddots & \dddots & \dddots\;\;
      \end{array}
    \right).
\end{align*}
Based on this decomposition, we can find the symbol. The symbol
of $A_k$ is obviously just
\begin{align*}
    \widehat{A_k}(\theta)&=\frac{h_k}{60}(-\ee^{-2\theta\ii}+4\ee^{-1\theta\ii}+24 +4\ee^{\theta\ii}-\ee^{2\theta\ii}).
\end{align*}
The symbol corresponding to $B_k$ is determined as follows:
\begin{align*}
    B_k \ul{\phi}_k(\theta) & = ( 2 (-1)^j \ee^{j\theta\ii} + (-1)^j ( \ee^{(j+2)\theta\ii} + \ee^{(j-2)\theta\ii} )   )_{j \in \mathbb{Z}} \\
	  & = ( 2 \ee^{j(\theta+\pi) \ii} + ( \ee^{2\theta\ii} + \ee^{-2\theta\ii} ) \ee^{j(\theta+\pi) \ii}   )_{j \in \mathbb{Z}} \\
	  & = \underbrace{(2+\ee^{2\theta\ii} + \ee^{-2\theta\ii})}_{\widehat{B_k}(\theta):=}  \ul{\phi}_k(\theta+\pi).
\end{align*}
So, we obtain
\begin{align*}
    M_k \ul{\phi}_k(\theta) = \widehat{A_k}(\theta)\ul{\phi}_k(\theta) + \widehat{B_k}(\theta)\ul{\phi}_k(\theta+\pi)
\end{align*}
and, as $\theta+2\pi \eqsim \theta$, also
\begin{align*}
    M_k \ul{\phi}_k(\theta+\pi) = \widehat{B_k}(\theta+\pi)\ul{\phi}_k(\theta) + \widehat{A_k}(\theta+\pi)\ul{\phi}_k(\theta+\pi).
\end{align*}
This shows, that $M_k$ does not preserve a one dimensional linear span anymore, but a two-dimensional span,
spanned by $\ul{\phi}_k(\theta)$ and $\ul{\phi}_k(\theta+\pi)$. This is similar to the coarse-grid operator in the last
section and in the last two subsections. So, the symbol is a representation of $M_k$ with respect to the basis
formed by these two vectors:
\begin{align*}
    \widehat{M_k}(\theta) &= \left(
	\begin{array}{cc}
	      \widehat{A_k}(\theta)&\widehat{B_k}(\theta)\\
	      \widehat{B_k}(\theta+\pi)&\widehat{A_k}(\theta+\pi)
	\end{array}
    \right).
\end{align*}
The symbol $\widehat{K_k}(\theta)$ of the stiffness matrix $K_k$ can be determined completely analogous.

Also the symbol of the prolongation operator can be determined similarly to the cases of the last
sections. However, we need four frequencies to be able to reconstruct a function on the coarse grid, so we
use the ansatz
$
      \phi_{k-1}(2\theta,x) = \sum_{j=0}^3 A_j \phi_{k}(\theta+ j\pi/2,x),
$
where it is again sufficient to consider the values on the nodes (midpoints and end points of the
intervals). This can be used determine the coefficients $A_0$, $A_1$, $A_2$ and $A_3$.

For all $\theta = [0,2\pi)$, the prolongation operator $P_{k-1}$ maps the linear
span, spanned by
\begin{equation}\label{eq:basis0a}
  \ul{\phi}_{k-1}(2\theta)\quad \mbox{and}\quad  \ul{\phi}_{k-1}(2\theta+\pi)
\end{equation}
to the linear span, spanned by
\begin{equation}\label{eq:basisa}
  \ul{\phi}_k(\theta),\quad \ul{\phi}_k(\theta+\pi/2),\quad \ul{\phi}_k(\theta+\pi) \quad \mbox{and}\quad  \ul{\phi}_k(\theta+3\pi/2),
\end{equation}
and the restriction operator $P_{k-1}^T$ maps the linear
span, spanned by~\eqref{eq:basisa}, to the linear span, spanned by~\eqref{eq:basis0a}.
So, the symbol $\widehat{P_{k-1}}$ is a $2\times4$-matrix, for details we refer to the Mathematica notebook.
Based on the symbols of the individual components, we can again compute $\widehat{\mc{G}_{k+1}}(\theta)$, the
symbol of the overall operator. The eigenvalues of this matrix are $0,0,\tfrac1{30}$ and
$\tfrac{1}{10}$, so we obtain $C_A=\tfrac1{10}$.

So, we have seen that the constant $C_A$ takes the value $\tfrac13$ for the Courant element,
the value $\tfrac25$ for the $P^2$-spline discretization and $\tfrac1{10}$ for the standard $P^2$
discretization.

This indicates that the standard $P^2$ discretization has the best approximation properties.
However, the standard $P^2$ discretization needs two degrees of freedom per element, while
the other two discretizations need, each, one degree of freedom per element.
By defining $\hat{h}_k$ to be the distance between two nodes, i.e.,
$\hat{h}_k=\tfrac12h_k$ for the standard $P^2$-discretization 
and $\hat{h}_k=h_k$ for the other two discretizations,
we can redefine the approximation error estimate as follows:
\begin{equation}\nonumber
	\|u -\Pi_k u\|_{L^2(\Omega)}^2\le \hat{C}_A \hat{h}_k^2 |u|_{H^1(\Omega)}^2.
\end{equation}
Here, we obtain $\hat{C}_A= \tfrac13$ for the Courant element and $\hat{C}_A=\tfrac25$
for both of the quadratic discretizations.

As we have already mentioned, an extension to two dimensions is possible, however the terms get much
more complicated. We refer to the complementary material, where we made an attempt to
generalize the analysis to two dimensions.

\section{Concluding remarks}
We have seen that the terms that are constructed using LFA can be
treated well using symbolic computation, particularly using CAD.
Moreover, we have seen that the method of LFA
can be applied in a wide range of problems. Besides is application to multigrid solvers,
which is well studied in literature, cf.~\cite{Brandt:1977,Brandt:1994,Trottenberg:2001}, LFA
can be applied to other problems occurring in numerical analysis, like the computation of
approximation error estimates.

\bibliographystyle{amsplain}
\bibliography{literature}

\end{document}